\documentclass{amsart}
\usepackage[centertags]{amsmath}
\usepackage{amsfonts}
\usepackage{amssymb}
\usepackage{amsthm}
\usepackage{setspace}
\usepackage{graphicx}
\usepackage[shortlabels]{enumitem}
\usepackage{mathrsfs}

\theoremstyle{plain}
\newtheorem{thm}{Theorem}[section]

\newtheorem{lem}[thm]{Lemma}

\theoremstyle{definition}

\numberwithin{equation}{section}

\begin{document}
	
	\title[The sharp bound of the Hankel determinant]{The sharp bound of the third Hankel determinant  for inverse of convex functions}
	
	\author[B. Rath, K. S. Kumar, D. V. Krishna]{Biswajit Rath$^{1}$, K. Sanjay Kumar $^{2},$ D. Vamshee Krishna$^{3}$}
	
	\address{$^{1.2.3}$Department of Mathematics,
		GITAM School of Science, GITAM (Deemed to be University),
		Visakhapatnam- 530 045, A.P., India}

	\email{brath@gitam.edu$^{1*}$,skarri9@gitam.in$^{2}$,vamsheekrishna1972@gmail.com$^{3}$}

	\begin{abstract}
		The objective of this paper is to find the best possible upper bound of the third Hankel determinant for inverse of convex functions.
	\end{abstract}
	
	
	
	\date{}
	
	\maketitle
	
	\section{Introduction}
	Denote $\mathcal{H}$ the family all analytic functions in unit disk $\mathbb{D}=\{z\in\mathbb{C}:|z|<1\}$ and $\mathcal{A}$ be the subfamily of functions $f$ normilized by the conditions \\$f(0)=f'(0)-1=0,$ i.e, of the type
	\begin{equation}
		f(z)=\sum_{n=1}^{\infty}a_{n}z^{n},~a_{1}:=1,
	\end{equation}
	and $\mathcal{S}$ be the subfamily of $\mathcal{A},$ possessing univalent (schlicht) mappings. For $f\in\mathcal{S},$ has an inverse $f^{-1}$ given by
	\begin{equation}
		f^{-1}(w)=w+\sum_{n=2}^{\infty}t_nw^n,~~|w|<r_o(f);\left(r_o(f)\geq\frac{1}{4}\right).
	\end{equation}
	A typical problem in geometric function theory is to study a functional made up of combination of the coefficients of the original functions. For the positive integers $r,~n,$ Pommerenke \cite{Pommerenke1966} characterized the $r^{th}$- Hankel determinant of $n^{th}$-order for $f$ given in (1.1), defined as follows:
	\begin{equation}
		H_{r,n}(f)=  \begin{array}{|cccc|}
			a_{n} & a_{n+1} & \cdots & a_{n+r-1} \\
			a_{n+1} & a_{n+2} & \cdots & a_{n+r} \\
			\vdots& \vdots & \ddots & \vdots \\
			a_{n+r-1} & a_{n+r} & \cdots & a_{n+2r-2}
		\end{array}.
	\end{equation}
	
	The problem of finding sharp estimates of the third Hankel determinant obtained for $r = 3$ and $n = 1$ in (1.2), given by
	\begin{equation}
		H_{3,1}(f):= \begin{array}{|ccc|}
			a_{1}=1 & a_{2} & a_{3}\\
			a_{2} & a_{3} & a_{4}\\
			a_{3} & a_{4} & a_{5}
		\end{array}= 2a_2a_3a_4-a_3^3-a_4^2+a_3a_5-a_2^2a_5,
	\end{equation}
	is technically much tough than $r=n=2$.\\
	
	In recent years, many authors are working on obtaining upper bounds ~~~(see \cite{Babalola2010,Kwon2019-1,Sim2021,Srivastava2021,Zaprawa2021}) and a few papers were devoted to the estimation of sharp upper bound to $H_{3,1}(f),$ for certain subclasses of analytic functions (see \cite{Banga2020,Kowalczyk2018-1,Kowalczyk2018-2,Kwon2019-2,Lecko2019,Ullah2021}).
	Recently Lecko et al. \cite{Kowalczyk2018-1} obtained the sharp bound for the class of convex function denoted as $\mathcal{S}^c$, defined by 
	\begin{equation}
		\text{Re}\left\{1+\frac{zf''(z)}{f'(z)}\right\}>0.
	\end{equation}
	Motivated by these results, in this paper we obtain sharp estimate for  $H_{3,1}(f^{-1})$ when $f\in\mathcal{S}^c$ as $1/36$.
	
	The collection $\mathcal{P},$ of all functions $p,$ each one called as Carath\'{e}odory function \cite{Duren1983} of the form,
	\begin{equation}
		p(z) = 1+\sum_{t=1}^{\infty}c_{t}z^{t},
	\end{equation}
	having a positive real part in $\mathbb{D}$. In view of (1.4) and (1.5), the coefficients of the functions in $\mathcal{S}^c$ can be expressed in terms of coefficients of functions in $\mathcal{P}$. We then obtain the upper bound of $H_{3,1}(f^{-1})$, buliding our analysis on the familiar formulas of coefficients $c_2$ (see, \cite[p. 166]{Pommerenke1975}), $c_3$ (see \cite{LibZlo1982,Libera1983}) and $c_4$ can be found in \cite{Lecko2019}.
	
	The foundation for proofs of our main results is the following lemma and we adopt the procedure framed through Libera and Zlotkiewicz \cite{Libera1983}.
	\begin{lem}
		If $p\in\mathcal{P},$ is of the form (1.5) with $c_1\geq0,$ such that $c_1\in[0,2]$ then
		\begin{equation*}
			\begin{split}
				&2c_2=c_1^2+\nu\mu,
				\\
				&4c_3=c_1^3+2c_1\nu\mu-c_1\nu\mu^2+2\nu\left(1-|\mu|^2\right)\rho,
				\\
				\text{and}\\
				&\begin{split}
					8c_4=~&c_1^4+3c_1^2\nu\mu+\left(4-3c_1^2\right)\nu\mu^2+c_1^2\nu\mu^3+4\nu\left(1-|\mu|^2\right)\left(1-|\rho|^2\right)\psi\\&+4\nu\left(1-|\mu|^2\right)\left(c_1\rho-c\mu\rho-\bar{\mu}\rho^2\right),
				\end{split}
			\end{split}
		\end{equation*}
		where $\nu:=4-c_1^2,$  for some $\mu$, $\rho$ and $\psi$ such that $|\mu|\leq 1$, $|\rho|\leq 1$ and $|\psi|\leq1$.
	\end{lem}
	\section{Main result}
\begin{thm}
	If $f\in$ $\mathcal{S}^c$, then
	\begin{equation*}
		\big|H_{3,1}(f^{-1})\big| \leq \frac{1}{36}
	\end{equation*} 
	and the inequality is sharp for $p_0(z)=(1+z^3)/(1-z^3)$.
\end{thm}
\begin{proof}
	For $f\in$ $\mathcal{S}^c$, there exists a holomorphic function $p \in\mathcal{P}$ such that
	\begin{equation}
		\left\{1+\frac{zf''(z)}{f'(z)}\right\}=p(z)\Leftrightarrow \left\{f'(z)+zf''(z)\right\}=p(z)f'(z)
	\end{equation}
	Using the series representation for $f$  and $p$ in (2.1), a simple calculation gives
	\begin{equation}
		\begin{split}
			&a_2=\frac{c_1}{2},~ a_3=\frac{c_1^2+c_2}{6},~a_4=\frac{1}{12}\left[\frac{1}{2}c_1^3+\frac{3}{2}c_1c_2+c_3\right]\\\text{and}~~ &a_5=\frac{1}{20}\left[\frac{1}{6}c_1^4+c_1^2c_2+\frac{1}{2}c_2^2+\frac{4}{3}c_1c_3+c_4\right]
		\end{split}
	\end{equation}
	Now from the defination (1.2), we have 
	\begin{equation}
		w=f(f^{-1})=f^{-1}(w)+\sum_{n=2}^{\infty}a_n(f^{-1}(w))^n.
	\end{equation}
	Further, we have
	\begin{equation}
		w=f(f^{-1})=w+\sum_{n=2}^{\infty}t_nw^n+\sum_{n=2}^{\infty}a_n(w+\sum_{n=2}^{\infty}t_nw^n)^n.
	\end{equation}
	Upon simplification, we obtain
	\begin{align}\nonumber
		(t_2+a_2)w^2+(t_3+2a_2t_2+a_3)w^3+(t_4+2a_2t_3+a_2t_2^2+3a_3t_2+a_4)w^4\\+(t_5+2a_2t_4+2a_2t_2t_3+3a_3t_3+3a_3t_2^2+4a_4t_2+a_5)w^5+......=0.
	\end{align}
	Equating the coefficients of like power in (2.5), upon simplification, we obtain
	\begin{equation}
		\begin{split}
			&t_2=-a_2; t_3=\{-a_3+2a_2^2\}; t_4=\{-a_4+5a_2a_3-5a_2^3\};\\&t_5=\{-a_5+6a_2a_4-21a_2^2a_3+3a_3^2+14a_2^4\}.
		\end{split}
	\end{equation}
	Using the values of $a_n (n=2,3,4,5)$ from (2.2) in (2.6), upon simplification, we obtain
	\begin{equation}
		\begin{split}
			&t_2=-\frac{c_1}{2},~t_3=\frac{1}{6}\left(2c_1^2-c_2\right),~t_4=\frac{1}{24}\left(-6c_1^3+7c_1c_2-2c_3\right)\\\text{and}~
			&t_5=\frac{1}{120}\left(-6c_4+22c_1c_3-46c_1^2c_2+7c_2^2+24c_1^4\right).
		\end{split}
	\end{equation}
	Now,
	\begin{align}
		H_{3,1}(f^{-1})&= \begin{array}{|ccc|}
			t_{1}=1 & t_{2} & t_{3}\\
			t_{2} & t_{3} & t_{4}\\
			t_{3} & t_{4} & t_{5}
		\end{array}~,
	\end{align}
	Using the values of $t_j,~ (j=2,3,4,5)$ from (2.7) in (2.8), it simplifies to give
	
	\begin{equation}
		\begin{split}
			H_{3,1}(f^{-1})=\frac{1}{8640}&\left[
		4c_1^6-24c_1^4c_2+12c_1^3c_3+39c_1^2c_2^2-44c_2^3+36c_1c_2c_3\right.\\&\left.-36c_1^2c_4-60c_3^2+72c_2c_4\right].
		\end{split}
	\end{equation}
	In view of (2.9), using the values of $c_2,~c_3~\text{and}~c_4$ from lemma 1.1, gives
	\begin{equation}
		\begin{split}
			24c_1^4c_2=&12\left[c_1^6+c_1^4\nu\mu\right];\\
			12c_1^3c_3=&3\left[c_1^6+2c_1^4\nu\mu-c_1^4\nu\mu^2+2c_1^3\nu(1-|\mu|^2)\rho\right]\\
			44c_2^3=&\frac{11}{2}\left[c_1^6+3c_1^4\nu\mu+3c_1^2\nu^2\mu^2+\nu^3\mu^3\right];\\
			39c_1^2c_2^2=&\frac{39}{4}\left[c_1^6+2c_1^4\nu\mu+c_1^2\nu^2\mu^2\right];
			\\
			36c_1c_2c_3=&\frac{9}{2}\left[c_1^6+3c_1^4\nu\mu+2c_1^2\nu^2\mu^2-c_1^4\nu\mu^2-c_1^2\nu^2\mu^3\right.\\&\left.+2\nu\left(c_1^3+c_1\nu\mu\right)\left(1-|\mu|^2\right)\rho\right];
			\\
			60c_3^2=&\frac{15}{4}\left[c^6+4c^4\nu\mu+4c^4\nu^2\mu^2-2c^4\nu\mu^2-4c^2\nu^2\mu^3+c^2\nu^2\mu^4\right.\\
			& \left.+4\nu(c^3+2c\nu\mu-c\nu\mu^2)(1-|\mu|^2)\rho+4\nu^2(1-|\mu|^2)^2\rho^2\right];
			\\
			72c_2c_4-36c_1^2c_4=& \frac{9}{2}\left[	c_1^4\nu\mu+3c_1^2\nu^2\mu^2+\left(4-3c_1^2\right)\nu^2\mu^3+c_1^2\nu^2\mu^4\right.\\&+4\nu^2c_1\mu\left(1-\mu\right)\left(1-|\mu|^2\right)\rho-4\nu^2\left(1-|\mu|^2\right)|\mu|^2\rho^2\\&\left.+4\nu^2\left(1-|\mu|^2\right)\left(1-|\rho|^2\right)\mu\psi
			\right].
		\end{split}
	\end{equation}
	Imputting the values from (2.10) in the  expression (2.9), after simplifying, we get
	\begin{equation}
		\begin{split}
			H_{3,1}(f^{-1})=\frac{1}{8640}&\left[
			\frac{3}{4}c_1^2\nu^2\mu^2-3c_1^2\nu^2\mu^3+\frac{3}{4}c_1^2\nu^2\mu^4-\frac{11}{2}\nu^3\mu^3+18\nu^2\mu^3\right.\\&-\left(3c_1\nu^2\mu+3c_1\nu^2\mu^2\right)\left(1-|\mu|^2\right)\rho-3\nu^2\left(5+|\mu|^2\right)\left(1-|\mu|^2\right)\rho^2\\&\left.+18\nu^2\mu\left(1-|\mu|^2\right)\left.(1-|\rho|^2\right)\psi\right].
		\end{split}
	\end{equation}
	Putting $u:=c_1$ and taking $\nu=\left(4-u^2\right)$ in (2.11), we obtain
	\begin{equation}
		\begin{split}
			H_{3,1}(f^{-1})=\frac{\left(4-u^2\right)^2}{8640}&\left[
			\frac{3}{4}u^2\mu^2+\frac{3}{2}u^2\mu^3+\frac{3}{4}u^2\mu^4-(4-u^2)\mu^3\right.\\&-3u\mu\left(1+\mu\right)\left(1-|\mu|^2\right)\rho-3\left(5+|\mu|^2\right)\left(1-|\mu|^2\right)\rho^2\\&\left.+18\mu\left(1-|\mu|^2\right)\left.(1-|\rho|^2\right)\psi\right].
		\end{split}
	\end{equation}
	Taking modulus on both sides of (2.12), using $|\mu|=v\in[0,1]$, $|\rho|=w\in[0,1]$, $c_1=u\in[0,2]$ and $|\psi|\leq 1$, we obtain
	\begin{equation}
		\bigg|H_{3,1}(f^{-1})\bigg|\leq\frac{\vartheta\left(u,v,w\right)}{8640},
	\end{equation}
	where $\vartheta:\mathbb{R}^3\rightarrow\mathbb{R}$ is defined as
	\begin{equation}
		\begin{split}
			\vartheta\left(u,v,w\right)=\left(4 - u^2\right)^2 &\left[
			\frac{3}{4}u^2v^2+\frac{3}{2}u^2v^3+\frac{3}{4}u^2v^4+\left(4-u^2\right)v^3\right.\\&+3uv\left(1+v\right)\left(1-v^2\right)w+3\left(5+v^2\right)\left(1-v^2\right)w^2\\&\left.+18v\left(1-v^2\right)\left.(1-w^2\right)\right]
		\end{split}
	\end{equation}
	Now, we are making an attempt to maximize the function $\vartheta\left(u,v,w\right)$ on \\$\Omega:=[0,2]\times[0,1]\times[0,1]$.\\
	{\bf A.} On the vertices of $\Omega$, from (2.14), we get
	\begin{equation*}
		\begin{split}
			&\vartheta\left(0,0,0\right)=\vartheta\left(2,0,0\right)=\vartheta\left(2,1,0\right)=\vartheta\left(2,0,1\right)=\vartheta\left(2,1,1\right)=0,\\&\vartheta\left(0,0,1\right)=240,~\vartheta\left(0,1,0\right)=\vartheta\left(0,1,1\right)=64.
		\end{split}
	\end{equation*}
	{\bf B.} On the edges of  $\Omega$, from (2.14), we have
	
	(i) For the edge $u=0,~v=0,~0<w<1,$ we obtain.
	\begin{equation*}
		\vartheta\left(0,0,w\right)=240 w^2\leq 240.
	\end{equation*}
	
	(ii) For the edge $u=0,~v=1,~0<w<1$, we obtain
	\begin{equation*}
		\vartheta\left(0,1,w\right)=64.
	\end{equation*}
	
	(iii) For $u=0,~w=0,~0<v<1$,
	\begin{equation*}
		\vartheta\left(0,v,0\right)=32 v (9 - 7 v^2)\leq192 \sqrt{\frac{3}{7}},,~\text{for}~v=\sqrt{2}.
	\end{equation*}
	
	(iv) For $u=0,~w=1,~0<v<1$,
	\begin{equation*}
		\vartheta\left(0,v,1\right)=240 - 192 v^2 + 64 v^3 - 48 v^4\leq240.
	\end{equation*}
	
	
	(v) For $v=0,~w=1,0<u<2$,
	\begin{equation*}
		\vartheta\left(u,0,1\right)=15 (4 - u^2)^2\leq240.
	\end{equation*}
	
	
	(vi) For the edges: $v=1,~w=0,0<u<2$ or $v=1,~w=1,0<u<2$, we have
	\begin{equation*}
		\vartheta\left(u,1,w\right)=(4 - u^2)^2 (4 + 2 u^2)\leq64.
	\end{equation*}
	
	
	(vii) For the edges:  $u=2,~v=0,~0<w<1$ or $u=2,~v=1,~0<w<1$ or
	
	$u=2,~w=0,~0<v<1$ or $c=2,~w=1,~0<v<1$ or $v=0,~w=0,~0<u<2$, 
	
	we obtain
	\begin{equation*}
		\vartheta\left(2,v,w\right)=0.
	\end{equation*}
	{\bf C.} Now, we consider the six faces of $\Omega$.
	
	(i) On the face $u=2$, from (2.14), we obtain $$\vartheta\left(2,v,w\right)=0.$$
	
	(ii) On the face $u=0,~v\in(0,1)~\text{and}~w\in(0,1)$ from (2.14), we get
	\begin{equation*}
		\begin{split}
			\vartheta\left(0,v,w\right)&=288 v - 224 v^3 + (240 - 288 v - 192 v^2 + 288 v^3 - 48 v^4) w^2\\&=288 v - 224 v^3 + 48 (5 - v) (-1 + v)^2 (1 + v) w^2\\&\leq288 v - 224 v^3 + 48 (5 - v) (-1 + v)^2 (1 + v)  \\&=240 - 192 v^2 + 64 v^3 - 48 v^4\leq240.
		\end{split}
	\end{equation*}
	
%
	
	(iii) On the face $v=0~u\in(0,2),~w\in(0,1)$, from (2.14), we obtain
	\begin{equation*}
		\vartheta\left(u,0,w\right)=15 (4 - u^2)^2 w^2\leq15 (4 - u^2)^2\leq240.
	\end{equation*}
	
%
	
	(iv) On the face $v=1,~u\in(0,2),~w\in(0,1)$, from (2.14), we observe that the
	
	function $\vartheta\left(u,1,w\right)$ is independent of $w$, from \textbf{B}(vi), we have $\vartheta\left(u,1,w\right)\leq240.$
	\\
	
	(v) On the face $w=0,~u\in(0,2),~v\in(0,1)$, from (2.14), we obtain
	\begin{equation*}
		\begin{split}
			\vartheta\left(u,v,0\right)=(4 - u^2)^2& \left(\frac{3 u^2 v^2}{4} + \frac{3 u^2 v^3}{2} + (4 - u^2) v^3 + \frac{3 u^2 v^4}{4} + 18 v (1 - v^2)\right)
			\\=(4 - u^2)^2& \left(18 v - 14 v^3 + u^2 \left(\frac{3 v^2}{4} + \frac{v^3}{2} + \frac{3 v^4}{4}\right)\right)\\\leq(4 - u^2)^2& \left(12 \sqrt{\frac{3}{7}}+ 2u^2\right)\leq192\sqrt{\dfrac{3}{7}},~u\in(0,2).
		\end{split}
	\end{equation*}
	
	
	(vi) On the face $w=1,$ in (2.14), we obtain
	\begin{align*}
		\vartheta\left(u,v,1\right)=(4 - u^2)^2& \bigg(
		\frac{3}{4}u^2v^2+\frac{3}{2}u^2v^3+\frac{3}{4}u^2v^4+(4-u^2)v^3\\&+3uv(1+v)(1-v^2)+3(5+v^2)\left(1-v^2\right)\bigg)
		\\:=g_3(u,v)&,~\text{with}~(u,v)\in\mathbb{R}^2.
	\end{align*}
	
Note that all real solutions (u,v) of the system of equation
	
	\begin{align*}
		\frac{\partial g_3}{\partial u}=\frac{3}{2} (-4 + u^2) &\left[8 (-1 + v) v (1 + v)^2 - 10 u^2 (-1 + v) v (1 + v)^2 \right.\\&\left.+u^3 v^2 (3 + 2 v + 3 v^2) - 4 u (-10 + 9 v^2 - 2 v^3 + 3 v^4)\right]=0
	\end{align*}

and

	\begin{align*}
		\frac{\partial g_3}{\partial v}&=\frac{3}{2} (-4 + u^2)^2 (-8 v (2 - v + v^2) + u^2 v (1 + v + 2 v^2) \\&+ 
		u (2 + 4 v - 6 v^2 - 8 v^3))=0
	\end{align*}

	by a numerical computation are the following
	\begin{equation*}
		(0,0),(-2.63625,-1.53087),(-1.0493,1.14045)~\text{and}~(\pm 2,x),x\in \mathbb{R}.
	\end{equation*}

	Therefore, $g_3$ has no critical point in $(0,2)\times(0,1).$
	
	{\bf D.} Now, consider the interior portion of $\Omega$ i.e. $(0,2)\times(0,1)\times(0,1).$\\
	
	Differentiating $\vartheta(u,v,w)$ partially with respect $w$, we obtain
	\begin{equation*}
		\begin{split}
		\frac{\partial \vartheta}{\partial w}=\frac{1}{2} (4 - u^2)^2& \left[60 w^2 + 3 v^2 (u^2 + 4 u w - 16 w^2) + 
		12 v (6 + u w - 6 w^2) \right.\\&\left.+ 3 v^4 (u^2 - 4 u w - 4 w^2) + 
		2 v^3 (-28 + u^2 - 6 u w + 36 w^2)\right]
	\end{split}
	\end{equation*}
	
	upon solving $\frac{\partial \vartheta}{\partial w}=0$, we get
	
	\begin{equation*}
		w_0=-\frac{u v (1 + v)}{2 (5 - v) (1 - v)}\notin(0,1)~\text{for}~(u,v)\in(0,2)\times(0,1)
	\end{equation*}
	
	Hence $\vartheta(u,v,w)$ has no critical point in the interior of $\Omega$.\\
	In review of cases \textbf{A}, \textbf{B}, \textbf{C} and \textbf{D}, we obtained 
	\begin{equation}
		\max\bigg\{\vartheta(u,v,w):u\in[0,2],v\in[0,1],~w\in[0,1]\bigg\}=240.
	\end{equation}
	From expression (2.13) and (2.15), we obtain
	\begin{equation}
		\Big|H_{3,1}(f^{-1})\Big|\leq\frac{1}{36}.
	\end{equation}
	For $p_0\in\mathcal{S}^c$, we obtain $t_2=t_3=t_5=0,~	t_4=1/6$, which follows the result.
\end{proof}
	\textbf{Data Availability:}~My manuscript has no associate data

\end{document}